\documentclass[11pt,letterpaper]{article}

\usepackage{amsfonts, amsmath, amssymb, amscd, amsthm, graphicx, color}

\newtheorem{thm}{Theorem}[section]
\newtheorem{cor}[thm]{Corollary}
\newtheorem{lem}[thm]{Lemma}
\newtheorem{prop}[thm]{Proposition}

\theoremstyle{definition}
\newtheorem{defn}[thm]{Definition}
\theoremstyle{remark}
\newtheorem{rem}[thm]{Remark}
\newtheorem{ex}[thm]{Example}

\usepackage{xcolor}

\usepackage{hyperref}
\hypersetup{linktocpage}

\hypersetup{colorlinks,
    linkcolor={red!50!black},
    citecolor={blue!80!black},
    urlcolor={blue!80!black}}
\usepackage{float}

\newfont{\eufm}{eufm10}

\renewcommand{\L}{{\mathcal L}}
\renewcommand{\d }{{\rm d} }

\newcommand{\e }{\varepsilon }
\renewcommand{\kappa }{\varkappa}

\newcommand{\diam}{{\rm diam\,}}

\newcommand{\SG}{{\mathcal S} (G)}

\newcommand{\ldot}{.}
\begin{document}

\title{Invariant random subgroups of groups acting \\ on hyperbolic spaces}
\author{D. Osin\thanks{This work was supported by the NSF grant DMS-1308961.}}
\date{}
\maketitle
\vspace*{-8mm}
\begin{abstract}
Suppose that a group $G$ acts non-elementarily on a hyperbolic space $S$ and does not fix any point of $\partial S$. A subgroup $H\le G$ is \emph{geometrically dense in $G$} if the limit sets of $H$ and $G$ coincide and $H$ does not fix any point of $\partial S$. We prove that every invariant random subgroup of $G$ is either geometrically dense or contained in the elliptic radical (i.e., the maximal normal elliptic subgroup of $G$). In particular, every ergodic measure preserving action of an acylindrically hyperbolic group on a Borel probability space $(X,\mu)$ either has finite stabilizers $\mu$-almost surely or otherwise the stabilizers are acylindrically hyperbolic $\mu$-almost surely.
\end{abstract}

\section{Introduction}
To every discrete group $G$, one can associate a topological dynamical system as follows. Let $\SG$ denote the space of all subgroups of $G$ with topology inherited from the Tichonoff product topology of $\{0,1\}^G$. Then $\SG$ is Hausdorff and compact; if $G$ is countable, $\SG$ is second countable and hence metrizable. The group $G$ acts on $\SG$ by conjugation, written $g.H=gHg^{-1}$ for $g\in G$ and $H\in \SG$. It is easy to check that conjugations are homeomorphisms of $\SG$.

\begin{defn}
An \emph{invariant random subgroup} of a group $G$ (abbreviated \emph{IRS}) is a Borel $G$-invariant probability measure on $\SG$.
\end{defn}

The term ``invariant random subgroup" was suggested in \cite{AGV}, but this notion has been studied for quite some time. For a brief history of the subject and a survey of some recent developments we refer to \cite{AGV,G}.

One can think of IRSs as generalizations of normal subgroups. Indeed every $N\lhd G$ gives rise to an atomic IRS, the Dirac measure concentrated at $N$. The study of IRSs is motivated by the following.

\begin{ex}\label{ex}
Let $G$ be a group acting by probability measure preserving (abbreviated p.m.p.) transformations on a Borel space $(X, \nu)$. For a point $x\in X$, we denote by $G_x$ the stabilizer of $x$. Auslander and Moore \cite{AM} proved that the map $f\colon X\to \SG$ given by $x\to G_x$ is measurable. Thus for every Borel subset $A\subseteq \SG$ we can define
\begin{equation}\label{mn}
\mu (A) =\nu (f^{-1}(A)).
\end{equation}
It is clear that $\mu $ is a probability measure on $\SG$ and is $G$-invariant since $gG_xg^{-1}=G_{gx}$ and the action of $G$ on $X$ preserves $\nu$. Thus $\mu $ is an IRS of $G$.
\end{ex}

Abert, Glasner and Virag \cite{AGV} proved that the this example is universal in the following sense: \emph{For every IRS $\mu $ of a countable group $G$, there exists a p.m.p. action of $G$ on a Borel space $(X, \nu)$ such that $\mu $ and $\nu$ are related by (\ref{mn})}. Thus studying IRSs of $G$ is the same as studying stabilizers of p.m.p. actions of $G$.

The main goal of this paper is to study IRSs of groups acting on hyperbolic spaces. We briefly recall necessary terminology here and refer to Section 3 for more details. All group actions on metric spaces considered in this paper are supposed to be isometric by default. Given an action of a group $G$ on a hyperbolic space $S$, we denote by $\Lambda (G)$ the limit set of $G$ on the Gromov boundary $\partial S$. The action of $G$ on $S$ is \emph{non-elementary} if $\Lambda (G)$ is infinite and \emph{of general type} if it is non-elementary and $G$ does not fix any point of $\partial S$. A subgroup $H\le G$ is called {\it elliptic} (with respect to the given action of $G$) if it has bounded orbits. The following result is likely known, but we could not find it in the literature.

\begin{prop}[Prop. \ref{EG}]\label{propeg}
Suppose that a group $G$ admits a non-elementary action on a hyperbolic space. Then there exists a maximal normal elliptic subgroup $E(G)$ of $G$ with respect to this action.
\end{prop}

Note that if the action of $G$ is elementary, such a maximal subgroup may not exist (see Example \ref{ex1}).  We call $E(G)$ the \emph{elliptic radical} of $G$. Further we say that a subgroup $H\le G$ is \emph{geometrically dense} with respect to a given general type action of $G$ on a hyperbolic space $S$ if $H$ does not fix any point of $\partial S$ and $\Lambda (H)=\Lambda(G)$.

\begin{ex}
Let $G$ be a non-elementary hyperbolic group acting on its Cayley graph with respect to some finite generating set. Then $E(G)$ is finite and
\begin{enumerate}
\item[(a)] every infinite normal subgroup of $G$ is geometrically dense;
\item[(b)] a quasi-convex subgroup $H\le G$ is geometrically dense iff it is of finite index.
\end{enumerate}
\end{ex}

Recall that an IRS $\mu $ of $G$ is \emph{ergodic} if so is the action of $G$ on $(\SG,\mu)$. That is, for every $G$-invariant subset $A\subseteq \SG$, we have $\mu(A)=0$ or $\mu(A)=1$. We also say that an \emph{IRS $\mu$ has some property $P$} if $\mu$-almost every subgroup of $G$ has $P$. Our main goal is record the following fairly elementary (but seemingly useful) fact.

\begin{thm}\label{main}
Let $G$ be a countable group acting on a hyperbolic space $S$. Assume that the action of $G$ on $S$ is of general type and let $\mu $ be an ergodic IRS of $G$. Then either $\mu$ is geometrically dense or $\mu$ is contained in $E(G)$.
\end{thm}

The case of a closed non-elementary subgroup $G$ of a simple Lie group of $\mathbb R$-rank $1$ was previously considered in \cite[Proposition 11.3]{7s}; the result obtained there is stated in a slightly different way, but is easily seen to be equivalent to a particular case of Theorem \ref{main}. A similar theorem for groups acting on $CAT(0)$ spaces was also obtained in \cite{DGLL}.

We mention one particular application of Theorem \ref{main}. Recall that a group is \emph{acylindrically hyperbolic} if it admits a non-elementary acylindrical action on a hyperbolic space. (For more details, we refer to Section 4.2.) Examples of such groups include non-elementary hyperbolic and relatively hyperbolic groups, infinite mapping class groups of punctured surfaces, $Out(F_n)$ for $n\ge 2$, most $3$-manifold groups, finitely presented groups of deficiency at least $2$, and many other examples of interest, see \cite{DGO,MO,Osi15,Osi13} and references therein. Bowen, Grigorchuk, and Kravchenko \cite{BGK} proved that every acylindrically hyperbolic group has continuously many non-atomic ergodic IRSs.

We say that a subgroup $H$ of an acylindrically hyperbolic group $G$ is \emph{totally geometrically dense} if it is geometrically dense with respect to every non-elementary acylindrical action of $G$ on a hyperbolic space. Every acylindrically hyperbolic group $G$ contains a unique maximal normal finite subgroup, denoted $K(G)$ \cite[Theorem 2.24]{DGO}.

\begin{cor}\label{cor1}
Let $\mu$ be an ergodic IRS of a countable acylindrically hyperbolic group. Then either $\mu$ is totally geometrically dense or $\mu$ is supported on subgroups of $K(G)$; in particular, $\mu$ is acylindrically hyperbolic in the former case and finite in the latter case.
\end{cor}

This corollary can be reformulated in terms of actions as follows.

\begin{cor}\label{cor2}
Let $G$ be a countable acylindrically hyperbolic group. Suppose that $G$ acts ergodically by measure preserving transformations of a Borel probability space $(X,\nu)$. Then either the stabilizer of $\nu$-a.e. point of $X$ is acylindrically hyperbolic or the stabilizer of $\nu$-a.e. point of $X$ is finite. If, in addition, $K(G)=1$, the action is essentially free in the latter case.
\end{cor}

In particular, this corollary implies that if $G$ is an acylindrically hyperbolic group and $K(G)=1$, then every p.m.p. action of $G$ with ``small" stabilizers (e.g., amenable) is essentially free. This statement can be made precise using the notion of a small subgroup introduced in \cite{HO}.

The proof of our main theorem consists of two steps. First we observe that every IRS of a countable group satisfies a purely algebraic normality-like condition (Proposition \ref{rec}), and then show that this condition implies the dichotomy as in Theorem \ref{main}. In fact, the first step can be easily extracted from \cite{GG} and the author is grateful to Yair Glasner for explaining details of \cite{GG}. The author is also grateful to the referee for careful reading of the manuscript and useful comments.

\section{Recurrent subgroups and dynamical systems}

The following notion plays the central role in our paper.

\begin{defn}\label{dr}
We say that a subgroup $H$ of a group $G$ is \emph{recurrent} if for every $g\in G$, $H$ is a recurrent point of the  dynamical system $(\SG, g)$. That is, for every open neighborhood $\mathcal U\subseteq \SG$ of $H$, there are infinitely many $n\in \mathbb N$ such that $g^n\ldot H=g^nHg^{-n}\in \mathcal U$.
\end{defn}

It is useful to reformulate this definition in a purely algebraic form.

\begin{lem}\label{lr}
A subgroup $H$ of a group $G$ is recurrent if and only if for every $g\in G$ and every finite subset $\mathcal F\subseteq G$, there exists at least one (equivalently, infinitely many) $n\in \mathbb N$ such that $H\cap \mathcal F= g^nHg^{-n}\cap \mathcal F$.
\end{lem}

\begin{proof}
By the definition of the topology of $\SG$, the base of neighborhoods of $H$ is formed by the subsets
$$
\mathcal U_{\mathcal F} (H)=\{ K\le G \mid K\cap \mathcal F = H\cap \mathcal F\},
$$
where $\mathcal F$ ranges in the set of all finite subsets of $G$. Clearly $H$ is recurrent if and only if Definition \ref{dr} holds for neighborhoods from this base. It remains to note that the condition $g\ldot H=g^nHg^{-n}\in \mathcal U_{\mathcal F}(H)$ is equivalent to $g^nHg^{-n}\cap \mathcal F = H\cap \mathcal F$.
\end{proof}

Throughout this paper we write $x^y$ for $y^{-1}xy$.

\begin{cor}\label{cr}
Let $H$ be a recurrent subgroup of a group $G$. Then for every $g\in G$ and every $h\in H$, there exist infinitely many $n\in \mathbb N$ such that $h^{g^n}\in H$.
\end{cor}

\begin{proof}
Apply Lemma \ref{lr} to $\mathcal F=\{ h\}$.
\end{proof}

\begin{ex}
\begin{enumerate}
\item[(a)] If $G$ is a torsion group, then every subgroup of $G$ is recurrent. To make this notion more useful in the torsion case, one can state Definition \ref{dr} using arbitrary infinite subsemigroups of $G$ instead of cyclic subgroups $\langle g\rangle$. Proposition \ref{rec} proved below holds for this definition as well. However we do not need this in our paper.
\item[(b)] Suppose that the normalizer of $H$ has finite index in $G$. Then $H$ is recurrent. The converse is false in general, but is true, for example, if $G$ is polycyclic. We leave this as an exercise for the reader.
\end{enumerate}
\end{ex}

The main result of this section is the following. As we already mentioned in the introduction, it can be extracted from \cite{GG} although it is not stated explicitly in this form there.

\begin{prop}\label{rec}
Let $G$ be a countable group and let $\mu$ be an IRS of $G$. Then $\mu$-a.e. subgroup $H\le G$  is recurrent.
\end{prop}

\begin{proof}
Let $\mu $ be an IRS on $G$. For an element $g\in G$, let $R_g$ denote the set of recurrent points of the dynamical system $(\SG, g)$. Since the space $\SG$ is second countable, we can apply the topological version of the Poincar\'e recurrence theorem, which states that $\mu (R_g)=1$. Let $R=\bigcap_{g\in G} R_g$. Clearly $R$ is the set of all recurrent subgroups of $G$. Since $G$ is countable and $\mu$ is countably additive, we have $\mu(R)=1$.
\end{proof}

\section{Elliptic radical of groups acting on hyperbolic spaces}

\paragraph{3.1. Hyperbolic spaces and group actions.}
All actions of groups on metric spaces are supposed to be isometric by default. We first recall necessary definitions and properties of groups acting on hyperbolic spaces. Although many existent proofs of the results mentioned below assume properness of the space, they also hold in the general case; we refer to Sections 8.1-8.2 in \cite{Gro} or to \cite{H} for complete proofs in a more general context.

Let $(S,d)$ be a (not necessarily proper) hyperbolic space. The \emph{Gromov product} of two points $x,y\in S$ with respect to a point $z\in S$ is defined by
$$
(x,y)_z=\frac12( \d(x,z) +\d (y,z) -\d (x,y)).
$$

The next lemma is well-known (see \cite{Gro} or Proposition 21 in \cite[Chapter 2]{GH}).

\begin{lem}\label{8d}
Let $S$ be a $\delta$-hyperbolic space. Then for any $x,y,z,t\in S$, we have
$$
(x,z)_t\ge \min \{ (x,y)_t, (y,z)_t\} -8\delta.
$$
\end{lem}

By $\partial S$ we denote the \emph{Gromov boundary} of $S$, which is defined as the set of equivalence classes of sequences converging at infinity. More precisely, a sequence $(x_n)$ of elements of $S$ \emph{converges at infinity} if $(x_i, x_j)_s\to \infty $ as $i,j\to \infty$; this is independent of the choice of the base point $s\in S$. Two such sequences $(x_i)$ and $(y_i)$ are \emph{equivalent} if $(x_i,y_j)_s\to \infty$ as $i,j\to \infty$. If $x$ is the equivalence class of $(x_i)$, we say that the sequence $x_i$ \emph{converges} to $x$. This naturally defines a  topology on $\widehat S=S\cup \partial S$ extending the topology on $S$ so that $S$ is dense in $\widehat S$.

\begin{rem}\label{rem}
It is useful to remember the following elementary fact: if $(x_i)$, $(y_i)$ are two sequences of points in $S$ converging to infinity such that $\sup_i\d(x_i, y_i) <\infty$, then $(x_i)$ and $(y_i)$ converge to the same point of $\partial S$.
\end{rem}

We denote by $\Lambda (G)$ the set of limit points of $G$ on $\partial S$. That is, $$\Lambda (G)=\overline{Gs}\cap \partial S,$$ where $s\in S$ and $\overline{Gs}$ is the closure of the $G$-orbit of $s$ in $\widehat S$. Using Remark \ref{rem} it is easy to show that this definition is independent on the choice of a particular orbit.

It is also easy to check that if a sequence $(x_i)$ converges to a point $x\in \partial S$, then for every $g\in G$, the sequence $(gx_i)$ also converges to a point of $\partial S$, which only depends on $x$ and $g$. This allows one to define an action of $G$ on $\partial S$, which turns out to be continuous.

An element $g\in G$ is called \emph{elliptic} if it fixes (setwise) a bounded subset of $S$. An element $g\in G$ is  \emph{loxodromic} if the map $\mathbb Z\to S$ defined by $n\mapsto g^ns$ is a quasi-isometric embedding for every $s\in S$ (here we assume that $\mathbb Z$ is equipped with the standard metric). Equivalently, $g$ is loxodromic if it is not elliptic and fixes exactly two distinct points $g^+$ and $g^-$ on $\partial S$.

For every loxodromic element $g\in G$, we have
$$
\lim\limits_{n\to \infty} g^ns= g^+ \;\; \forall s\in \widehat S\setminus \{ g^-\}\;\;\;\;{\rm and} \;\;\;\; \lim\limits_{n\to \infty} g^{-n}s= g^- \;\; \forall s\in \widehat S\setminus \{ g^+\}
$$
as $n\to \infty$. This is called the \emph{north-south dynamics} of the action of a loxodromic element. In particular, we have $\Lambda (\langle g\rangle)=\{ g^+, g^-\}$.

Two loxodromic elements $g,h\in G$ are called {\it independent} if $\{ g^+, g^-\}\cap \{ h^+, h^-\} = \emptyset$. We denote by $\L(G)$ the set of all loxodromic elements of $G$ and let $$\mathcal H(G)=\{ g^\pm \mid g\in \L(G)\}.$$ We also denote by $Fix(G)$ the set of fixed points of $G$ on $\partial S$.

Possible actions of groups on hyperbolic spaces can be classified as follows according to the cardinality of $\Lambda (G)$ (see \cite[Sections 8.1-8.2]{Gro}, \cite{H},  or \cite[Section 3a]{CCMT})

\begin{enumerate}
\item[1)] (\emph{elliptic action})\; $|\Lambda (G)|=0$. Equivalently,  $G$ has bounded orbits.
\item[2)] (\emph{parabolic action})\;  $|\Lambda (G)|=1$. Equivalently, $G$ has unbounded orbits and contains no loxodromic elements.  In this case $Fix(G)=\Lambda(G)$.
\item[3)] (\emph{lineal action})\; $|\Lambda (G)|=2$. Equivalently, $G$ contains a loxodromic element and any two loxodromic elements have the same limit points on $\partial S$. In this case $Fix(G)\subseteq \Lambda(G)$.

\item[4)] $|\Lambda (G)|=\infty$. Then $G$ always contains loxodromic elements. In turn, this case breaks into two subcases.
\begin{enumerate}
\item[a)] (\emph{quasi-parabolic action}) $Fix(G)\ne \emptyset$. Then $G$ fixes a unique point of $\partial S$.
\item[b)] (\emph{general type action}) $Fix(G)=\emptyset$. Equivalently, $G$ contains at least $2$ (or infinitely many) independent loxodromic elements.
\end{enumerate}
\end{enumerate}

The action of $G$ is called \emph{elementary}  in cases 1)--3) and non-elementary in case 4).

The following can be found in \cite[Section 8.2]{Gro}; alternatively, see \cite[Theorems 2.6, 2.9]{H}.

\begin{lem}\label{dense}
Suppose that a group $G$ acts non-elementarily on a hyperbolic space $S$. Then $\mathcal H(G)$ is dense in $\Lambda (G)$. Moreover, if the action is of general type, then for any two points $x,y\in \Lambda (G)$ and any open neighborhoods $A,B\subseteq \partial S$ of $x$ and $y$, respectively, there exists a loxodromic element $g\in G$ such that $g^+\in A$ and $g^-\in B$.
\end{lem}

\paragraph{3.2. Elliptic radical.}
Given an action of a group $G$ on a hyperbolic space $S$, we define the \emph{elliptic radical} of $G$ with respect to this action by
\begin{equation}\label{defEG}
E(G)=\{ g\in G\mid gx=x \; \forall\, x\in \Lambda (G)\}.
\end{equation}

Note that this definition makes sense for arbitrary actions on hyperbolic spaces. For non-elementary actions, the following proposition provides an equivalent characterization.

\begin{prop}\label{EG}
Suppose that a group $G$ admits a non-elementary action of general type on a hyperbolic space $S$. Then $E(G)$ is the unique maximal elliptic normal subgroup of $G$.
\end{prop}

\begin{proof}
Obviously $E(G)$ is normal in $G$. Further, let $N$ be a normal elliptic subgroup of $G$. Fix some $s\in S$. Since $N$ is elliptic, the diameter  $D=\diam (Ns)$ is finite. For every $a\in N$ and every sequence $(g_i)\subseteq G$ such that $\lim\limits_{i\to \infty}g_is=x\in \partial S$ for some $s\in S$, we have
$$
\d (ag_is,g_is)=\d (g^{-1}_iag_is, s)\le D
$$
since $g^{-1}_iag_i\in N$. By Remark \ref{rem}, this implies that $ax=x$ and thus $N\le E(G)$.

Finally we note that by the classification of group actions on hyperbolic spaces, only elliptic subgroups can fix more that $2$ points of $\partial S$. Since the action of $G$ is non-elementary we have $|Fix(E(G))|\ge |\Lambda (G)|=\infty $ and thus $E(G)$ is elliptic.
\end{proof}

\begin{ex}\label{ex1}
In general, the proposition may fail in various ways for elementary (lineal and parabolic) actions. Indeed, for any lineal action, $E(G)$ as defined by (\ref{defEG}) contains loxodromic elements and hence is not elliptic. Furthermore, the maximal normal elliptic subgroup of $G$ may not exist for parabolic actions. Indeed, let
$$
W=\langle a,t\mid a^2=1,\; [a^{t^i}, a] =1,\, \forall i\in \mathbb Z\rangle =\mathbb Z_2 \wr \mathbb Z.
$$
It is easy to see that $W$ is an ascending HNN-extension of the locally finite subgroup $A=\langle a_i, \, i\in \mathbb N\rangle $, where $a_i=a^{t^i}$, associated to a monomorphism $A\to A$ given by $a_i\mapsto a_{i+1}$. Let $T$ be the corresponding Bass-Serre tree and let $G=\langle a_i, \, i\in \mathbb Z\rangle $. Every finitely generated subgroup of $G$ is normal (since $G$ is abelian) and finite (hence elliptic with respect to the action on $T$). However $G$ has an unbounded orbit in $T$ and hence its action is parabolic. Thus the maximal normal elliptic subgroup of $G$ does not exist.
\end{ex}

\section{A dichotomy for IRSs of groups acting on hyperbolic spaces}

Let $G$ be a group acting on a hyperbolic space $(S,\d)$. Throughout this section we assume that $G$ is countable and the action is of general type.

\paragraph{4.1. Recurrent subgroups of groups acting on hyperbolic spaces}
We will need the following result from \cite[Lemma 3.5]{Osi13}.

\begin{lem}\label{gh}
Suppose that for some $a,b\in G$ there exist $x,y\in S$ and $C>0$ such that
\begin{equation}\label{max}
\max\{ \d (x,ax), \d(y,by)\} \le C
\end{equation}
and
\begin{equation}\label{min}
\min \{ \d (x,bx) , \d (y, ay)\} \ge \d (x,y) + 3C.
\end{equation}
Then $ab$ is a loxodromic isometry.
\end{lem}

Recall that for a point $a\in \partial S$, $G_a$ denotes the stabilizer of $a$.

\begin{lem}\label{Eg}
Suppose that a group $G$ acts on a hyperbolic space $S$. Let us fix some $s\in S$ and let $g\in G$ be a loxodromic element. Then for every $f\in G\setminus G_{g^+}$, there exists a constant $D$ such that
\begin{equation}\label{fgn1}
|\d (fg^ns, g^ns)-2\d (g^ns,s)|\le D
\end{equation}
for all $n\in \mathbb N$.
\end{lem}

\begin{proof}
Since $f\notin G_{g^+}$, the sequences $(g^ns)$ and $(fg^ns)$ converge to distinct points of $\partial S$. It follows that the Gromov products $(fg^ns, g^ns)_s$ are uniformly bounded by some constant $C$. Therefore, we obtain
\begin{equation}\label{fgn2}
\begin{array}{rcl}
\d (fg^ns, g^ns) & = & \d (fg^ns, s)+\d(g^ns,s) - 2(fg^ns, g^ns)_s \ge \\&&\\
&& \d(fg^ns,fs) - \d (fs,s) +\d(g^ns,s) - 2C= \\&&\\
&&2 \d(g^ns,s) -\d(fs,s) -2C.
\end{array}
\end{equation}

On the other hand, we obviously have
\begin{equation}\label{fgn3}
\d (fg^ns, g^ns) \le \d (fg^ns, fs)+\d (fs,s) +\d(s,g^ns) \le  2 \d(g^ns,s) +\d(fs,s).
\end{equation}
Combining (\ref{fgn2}) and (\ref{fgn3}), we obtain (\ref{fgn1}) for $D=\d(fs,s) +2C$.
\end{proof}

\begin{lem}\label{com}
Let $g\in \L (G)$ and let $f\in G$ be an element such that $fg^+ \ne g^+$ and $fg^- \ne g^-$. Then for all sufficiently large $n\in \mathbb N$, the commutator $[f, g^n]=fg^nf^{-1}g^{-n}$ is loxodromic.
\end{lem}

\begin{proof}
Let us fix some $x\in S$ and let $C=\d (fx,x)$. Let also $y=g^nx$. By Lemma \ref{gh}, it suffices to show that the inequalities (\ref{max}) and (\ref{min}) are satisfied for the elements $a=f$ and $b=g^nf^{-1}g^{-n}$.

Note that $\d (y, by)=\d (g^nx, g^nf^{-1}x)=\d(x,f^{-1}x)=C$ and thus the inequality (\ref{max}) holds. Since $fg^+\ne g^+$, we can apply Lemma \ref{Eg}, which tells us that  $\d (y, ay)=\d (g^nx, fg^nx) $ grows at least as $2\d (x,y)$ minus a constant as $n\to \infty$ (note the multiple $2$ here). Since $g$ is loxodromic, we have $\d (x, y)\to \infty$ as $n\to \infty$. In particular, if $n$ is  sufficiently large, the inequality $\d (y, ay)\ge \d (x,y) + 3C$ holds. Similarly we show that $\d (x,bx)\ge \d (x,y) + 3C$ whenever $n$ is  sufficiently large; we need the assumption $fg^{-}\ne g^{-}$ here.
\end{proof}

\begin{lem}\label{loxo}
Let $H\le G$ be a recurrent subgroup. If $H$ acts non-trivially on $\Lambda (G)$, then $H$ contains a loxodromic element.
\end{lem}

\begin{proof}
Assume that some $f\in H$ acts nontrivially on $\Lambda (G)$. Let $x\in \Lambda (G)$ be a point such that $fx\ne x$. Let $y=fx$. We fix two disjoint open neighborhoods $U$ and $V$ of $x$ and $y$, respectively, such that $fU\subseteq V$. By Lemma \ref{dense}, we can find a loxodromic element $g\in G$ such that $g^\pm \subseteq U$. Then $fg^{\pm}\subseteq fU\subseteq V$. In particular, we have $fg^\pm \ne g^\pm $. By Lemma \ref{com}, the commutator $[f,g^n]$ is loxodromic for all sufficiently large $n\in \mathbb N$. Since $H$ is recurrent, we conclude that there exist arbitrarily large $n$ such that $g^{n}f^{-1}g^{-n}\in H$ and hence $[f,g^n]\in H$. \end{proof}

We are now ready to prove the main result of this section.

\begin{thm}\label{main1}
Suppose that a countable group $G$ admits a general type action on a hyperbolic space $S$. Then every  recurrent subgroup $H$ of $G$ is either geometrically dense in $G$ with respect to this action or belongs to $E(G)$. In particular, in the former case the action of $H$ on $S$ is of general type and in the latter case $H$ is elliptic.
\end{thm}

\begin{proof}
If $H$ acts trivially on $\Lambda (G)$, then $H\le E(G)$. Thus we can assume that the action of $H$ on $\Lambda (G)$ is nontrivial. We first prove that $\Lambda (H)=\Lambda (G)$ in this case.

Since limit sets are closed, it suffices to show that every open neighborhood $U\subseteq \widehat S$ of every $x\in \Lambda (G)$ contains an element of $\Lambda (H)$. By Lemma \ref{loxo}, there exists loxodromic $h\in H$. If $h^+ \in U$ or $h^-\in U$, we are done. Thus we can assume that $\{ h^+, h^-\}\cap U=\emptyset$. By Lemma \ref{dense}, there exists a loxodromic element $g\in G$ such that $g^+\in U$. In particular, we have $h\notin G_{g^+}$. Let us fix some $s\in S$. We claim that
\begin{equation}\label{gnhgn}
\lim_{n\to\infty} h^{g^{n}}s= g^- .
\end{equation}
Indeed using Lemma \ref{Eg} we obtain
\begin{equation}\label{gns}
\begin{array}{rcl}
2(h^{g^{n}}s, g^{-n}s)_s &= &\d(g^{-n}hg^{n}s,s) + \d(g^{-n}s,s) -\d (g^{-n}hg^{n}s, g^{-n}s)=\\&&\\
&& \d(hg^{n}s,g^{n}s) + \d(s,g^{n}s) -\d (g^{n}s, h^{-1}s)\ge \\&&\\
&& 3\d (s, g^ns) -D - (\d (g^ns, s) +\d (s,h^{-1}s))    = \\&&\\
&& 2 \d (s, g^ns) -D - \d (hs,s) \to \infty
\end{array}
\end{equation}
as $n\to \infty$.
Since the sequence $(g^{-n}s)$ is convergent, we also have
\begin{equation}\label{gmn}
(g^{-n}s, g^{-m}s)_s \to \infty
\end{equation}
as $m,n\to \infty$. Combining (\ref{gns}), (\ref{gmn}), and Lemma \ref{8d}, we obtain that $(h^{g^{n}}s)$ converges to infinity and is equivalent to $(g^{-n}s)$. Hence we obtain (\ref{gnhgn}). It remains to note that by Corollary \ref{cr}, $(h^{g^{n}})$ has an infinite subsequence consisting of elements of $H$. Hence $g^-\in \Lambda(H)$ and thus $U\cap \Lambda(H)$ is non-empty.

Finally we note that $H$ cannot fix a point $a\in\partial S$. Indeed since the action of $G$ is of general type, we can find a loxodromic element $g\in G$ such that $g^{\pm}\notin \{a, h^+, h^-\}$. Using the north-south dynamics of the action of $g$, we conclude that $g^nh^{\pm}\ne a$ for sufficiently large $n$. On the other hand, the recurrency condition implies that there are arbitrarily large $n$ such that $g^nhg^{-n}\in H$. It is clear that $g^nhg^{-n}$ is loxodromic and its fixed points are exactly $g^nh^{\pm}$. Thus $H$ contains an element that does not fix $a$.
\end{proof}

\begin{proof}[Proof of Theorem \ref{main}]
The theorem follows from Proposition \ref{rec} and Theorem \ref{main1} immediately.
\end{proof}

\paragraph{4.2. An application to acylindrically hyperbolic groups.}
Recall that an isometric action of a group $G$ on a metric space $(S,\d)$ is {\it acylindrical} if for every $\e>0$ there exist $R,N>0$
such that for every two points $x,y$ with $\d (x,y)\ge R$, there are at most $N$ elements $g\in G$ satisfying
$$
\d(x,gx)\le \e \;\;\; {\rm and}\;\;\; \d(y,gy) \le \e.
$$
The notion of acylindricity goes back to Sela's paper \cite{Sel}, where it was studied for groups acting on trees. In the context of general metric spaces, this concept is due to Bowditch \cite{Bow}. Further, a group $G$ is called \emph{acylindrically hyperbolic} if it admits a non-elementary acylindrical action on a hyperbolic space. For details and recent developments in the study of acylindrically hyperbolic groups we refer to \cite{Osi13}.

\begin{proof}[Proof of Corollary \ref{cor1}]
Let $G$ be a group acting non-elementarily and acylindrically on a hyperbolic space $S$. By \cite[Lemma 6.15]{DGO} and part (a) of Proposition \ref{EG}, $E(G)=K(G)$ is finite in this case. Thus for every IRG $\mu$ of $G$, $\mu$-almost every subgroup of $G$ is geometrically dense with respect to the action on $S$ or belongs to $K(G)$. It is clear that the set of subgroups of $K(G)$ is $G$-invariant, and thus if $\mu$ is ergodic it must have measure $0$ or $1$. This implies the first claim of the corollary. The second part follows immediately from the fact that the action of a geometrically dense subgroup is always non-elementary and the fact that $K(G)$ is finite.
\end{proof}

Corollary \ref{cor2} is an obvious reformulation of Corollary \ref{cor1} (see Example \ref{ex}).

\vspace{1cm}

\noindent \textbf{Denis Osin: } Department of Mathematics, Vanderbilt University, Nashville 37240, U.S.A.\\
E-mail: \emph{denis.v.osin@vanderbilt.edu}


\begin{thebibliography}{99}

\bibitem{7s}
M. Abert, N. Bergeron, I. Biringer, T. Gelander, N. Nikolov, J. Raimbault, I. Samet, On the growth of $L^2$-invariants for sequences of lattices in Lie groups, \emph{arXiv:1210.2961v3.}


\bibitem{AGV}
M. Abert, Y. Glasner, B. Virag, Kesten's theorem for Invariant Random Subgroups, \emph{Duke Math. J.} \textbf{163}, no. 3 (2014), 465-488.

\bibitem{AM}
L. Auslander, C. C. Moore, Unitary representations of solvable Lie groups, \emph{Memoirs AMS} (1966), 66-77.


\bibitem{Bow}
B. Bowditch, Tight geodesics in the curve complex, \emph{Invent. Math.} \textbf{171} (2008), no. 2, 281-300.

\bibitem{BGK}
L. Bowen, R. Grigorchuk, R. Kravchenko,
Characteristic random subgroups of geometric groups and free abelian groups of infinite rank, arXiv:1402.3705.

\bibitem{CCMT}
P-E. Caprace, Y. Cornulier, N. Monod, R. Tessera, Amenable hyperbolic groups,  \emph{arXiv:1202.3585.}

\bibitem{DGLL}
B. Duchesne, Y. Glasner, N. Lazarovich, J. L\'ecureux, Geometric density for invariant random subgroups of groups acting on $CAT(0)$ spaces, arXiv:1409.8007.

\bibitem{DGO}
F. Dahmani, V. Guirardel, D. Osin, Hyperbolically embedded subgroups and rotating families in groups acting on hyperbolic spaces,  \emph{ arXiv:1111.7048}; \emph{Memoirs AMS}, to appear.

\bibitem{GG}
Y. Glasner, Invariant random subgroups of linear groups (with an appendix by Y. Glasner and T. Gelander), arXiv:1407.2872.

\bibitem{G}
T. Gelander, A lecture on Invariant Random Subgroups, \emph{arXiv:1503.08402}.

\bibitem{GH}
E. Ghys, P. de la Harpe, Sur les groupes hyperboliques d'apr\`es Mikhael Gromov. Progress in Mathematics, \textbf{83}. Birkh\"auser Boston, Inc., Boston, MA, 1990.


\bibitem{Gro}
M. Gromov, {\it Hyperbolic groups,} Essays in Group Theory, MSRI
Series, Vol.8, (S.M. Gersten, ed.), Springer, 1987, 75-263.

\bibitem{GH}
E. Ghys, P. de la Harpe, Sur les groupes hyperboliques d'apr\`es Mikhael Gromov. Progress in Mathematics, \textbf{83}. Birkh\"auser Boston, Inc., Boston, MA, 1990.

\bibitem{H}
M. Hamann, Group actions on metric spaces: fixed points and free subgroups, arXiv:1301.6513.

\bibitem{HO}
M. Hull, D. Osin, Transitivity degrees of countable groups and acylindrical hyperbolicity,  arXiv:1501.04182.

\bibitem{MO}
A. Minasyan, D. Osin, Acylindrical hyperbolicity of groups acting on trees, \emph{Math. Ann.} \textbf{362} (2015), no. 3-4, 1055-1105.

\bibitem{Osi15}
D. Osin, On acylindrical hyperbolicity of groups with positive first $\ell^2$-Betti number, \emph{Bull. Lond. Math. Soc.} \textbf{47} (2015), no. 5, 725-730.

\bibitem{Osi13}
D. Osin, Acylindrically hyperbolic groups, \emph{Trans. AMS}, to appear; \emph{arXiv:1304.1246}.

\bibitem{Sel}
Z. Sela, Acylindrical accessibility for groups,
\emph{Invent. Math.} \textbf{129} (1997), no. 3, 527-565.
\end{thebibliography}
\end{document}